\documentclass[11pt]{amsart}
\usepackage{graphicx}
\usepackage{amscd}
\usepackage{amsmath}
\usepackage{amsthm}
\usepackage{amsfonts}
\usepackage{amssymb}
\usepackage{mathrsfs}
\usepackage{enumitem}
\usepackage{xcolor}
\usepackage[colorlinks, citecolor=blue, linkcolor=red, pdfstartview=FitB]{hyperref}
\usepackage[latin1]{inputenc}

\newcommand{\bC}{{\mathbb{C}}}

\newcommand{\bN}{{\mathbb{N}}}

  \newcommand{\A}{{\mathcal{A}}}

\renewcommand{\H}{{\mathcal{H}}}
  \newcommand{\I}{{\mathcal{I}}}

\renewcommand{\S}{{\mathcal{S}}}

\renewcommand{\phi}{\varphi}

\newcommand{\upchi}{{\raise.35ex\hbox{$\chi$}}}

\newcommand{\spn}{\operatorname{span}}

\newtheorem{lemma}{Lemma}[section]
\newtheorem{theorem}[lemma]{Theorem}

\newtheorem{corollary}[lemma]{Corollary}
\theoremstyle{definition}
\newtheorem{definition}[lemma]{Definition}

\begin{document}

\title{Faithfulness of bi-free product states}
\author{Christopher Ramsey}
%
%
%
%
%
%
%
%
%
%
%
%
%

\address{Department of Mathematics \\ University of Manitoba \\ Winnipeg, Manitoba, Canada}
\email{christopher.ramsey@umanitoba.ca}


\begin{abstract}
Given a non-trivial family of pairs of faces of unital C$^*$-algebras where each pair has a faithful state,
it is proved that if the bi-free product state is faithful on the reduced bi-free product of this family of pairs of faces then each pair of faces arises as a minimal tensor product. A partial converse is also obtained.
\end{abstract}

\maketitle

\section{introduction}

The reduced free product was given independently by Avitzour \cite{Avitzour} and Voiculescu \cite{Voic85} and it has been foundational in the development of free probability.
Dykema proved in \cite{Dykema} that the free product state on the reduced free product of unital C$^*$-algebras with faithful states is faithful. 
In consequence of this, if $\{\A_i\}_{i\in\I}$ is a free family of unital C$^*$-algebras in the non-commutative C$^*$-probability space $(\A, \varphi)$ and if $\varphi$ is faithful on $C^*(\{\A_i\}_{i\in\I})$ then 
\[
C^*(\{\A_i\}_{i\in\I}) \simeq *_{i\in \I} (\A_i, \varphi|_{\A_i}),
\]
the reduced free product of the $\A_i$'s with respect to the given states.
This can be deduced from a paper of Dykema and R{\o}rdam, namely \cite[Lemma 1.3]{DykemaRordam}.

The present paper is the result of the author's attempt to prove the same result in the new context of bi-free probability introduced by Voiculescu \cite{Voic14}. To this end, suppose $(\A_l^{(i)}, \A_r^{(i)})_{i\in\I}$ is a non-trivial family of pairs of faces  in the non-commutative C$^*$-probability space $(\A,\varphi)$. 
If $\varphi_i = \varphi|_{C^*(\A_l^{(i)}, \A_r^{(i)})}$ is faithful on $C^*(\A_l^{(i)}, \A_r^{(i)})$, for all $i\in \I$, then 
it will be proven that if the bi-free product state ${**}_{i\in\I} \varphi_i$ is faithful on the reduced bi-free product ${**}_{i\in \I} (\A_l^{(i)}, \A_r^{(i)})$ then $C^*(\A_l^{(i)}, \A_r^{(i)}) \simeq \A_l^{(i)} \otimes_{min} \A_r^{(i)}, i\in \I$. A converse is shown with the added asumption that each $\varphi_i$ is a product state. Moreover, in this case there is a commensurate result to that which follows from Dykema and R{\o}rdam, mentioned above.

It should be mentioned that the failure in general of the faithfulness of the bi-free product state has been pointed out in \cite{FreslonWeber} and this failure has been the cause of the introduction of weaker versions of faithfulness in the bi-free context \cite{FreslonWeber, Skoufranis}.
\vskip 6 pt
\noindent {\bf Acknowledgements:} The author would like to thank Scott Atkinson for sparking my interest into bi-free independence and for suggesting the reduced bi-free product, Paul Skoufranis for pointing out an error in a previous version of this paper, and the referee for their help in improving several difficult passages.

\section{Bi-free independence and the reduced bi-free product}

We will first take some time to recall the definition of bi-free independence from \cite{Voic14} and then define the reduced bi-free product of C$^*$-algebras and the bi-free product state.

Fix a non-commutative C$^*$-probability space $(\A, \varphi)$, that is a unital C$^*$-algebra and a state. Given a set $\I$, suppose that for each $i\in\I$ there is a pair of unital C$^*$-subalgebras $\A_l^{(i)}$ and $\A_r^{(i)}$ of $\A$, a ``left'' algebra and a ``right'' algebra. We call the set $(\A_l^{(i)}, \A_r^{(i)})_{i\in\I}$ a family of pairs of faces in $\A$. Such a family will be called {\em non-trivial} if $|\I| \geq 2$ and $C^*(\A_l^{(i)}, \A_r^{(i)}) \neq \bC$ for all $i\in \I$. That is, there are at least two pairs of faces and there are no trivial pairs of faces. 

 Let $(\pi_i, \H_i, \xi_i)$ be the GNS construction for $(C^*(\A_l^{(i)}, \A_r^{(i)}), \varphi_i)$ where $\varphi_i = \varphi|_{C^*(\A_l^{(i)}, \A_r^{(i)})}$. Voiculescu \cite{Voic14} (and even way back in \cite{Voic85}) observed that there are two natural representations of $B(\H_i)$ on the free product Hilbert space, which we will now introduce.
The free product Hilbert space,
\[
(\H, \xi) = *_{i\in \I} (\H_i, \xi_i),
\]
is given by associating all of the distinguished vectors and then forming a Fock space like structure. Namely, if $\mathring \H_j = \H_j \ominus \bC\xi_j$, then
\[
\H  := \bC\xi \oplus \bigoplus_{\begin{smallmatrix} n\in \bN \\ i_1,\cdots, i_n\in \I \\ i_1 \neq \cdots \neq i_n \end{smallmatrix}} \mathring \H_{i_1}\otimes \cdots \otimes \mathring \H_{i_n}. 
\]
To define these representations we need to first build some Hilbert spaces and some unitaries.
To this end, define 
\begin{align*}
\quad\quad \H(l,i) & := \bC\xi \oplus \bigoplus_{\begin{smallmatrix} n\in \bN \\ i_1,\cdots, i_n\in \I \\i\neq i_1 \neq \cdots \neq i_n \end{smallmatrix}} \mathring \H_{i_1}\otimes \cdots \otimes \mathring \H_{i_n} 
\quad \quad \textrm{and}
\\ \H(r,i) & := \bC\xi \oplus \bigoplus_{\begin{smallmatrix} n\in \bN \\ i_1,\cdots, i_n\in \I \\i_1 \neq \cdots \neq i_n \neq i \end{smallmatrix}} \mathring \H_{i_1}\otimes \cdots \otimes \mathring \H_{i_n}.
\end{align*}
Then there are unitaries $V_i : \H_i \otimes \H(l,i) \rightarrow \H$ and $W_i : \H(r,i) \otimes \H_i$ given by concatenation (with appropriate handling of $\xi_i$ and $\xi$).
Finally, the two natural representations are the left representation $\lambda_i : B(\mathcal H_i) \rightarrow B(\H)$ which is defined as 
\[
\lambda_i(T) = V_i(T \otimes I_{\H(l,i)})V_i^*
\]
and the right representation $\rho_i : B(\mathcal H_i) \rightarrow B(\H)$ which is defined as
\[
 \rho_i(T) = W_i(I_{\H(r,i)} \otimes T)W_i^*.
\]

With all of this groundwork established we can finally define bi-free independence. Note that $\check *$ below refers to the full (or universal) free product of C$^*$-algebras.

\begin{definition}[Voiculescu \cite{Voic14}]
The family of pairs of faces $(\A_l^{(i)}, \A_r^{(i)})_{i\in\I}$ in the non-commutative probability space $(\A, \varphi)$ is said to be {\em bi-freely independent} with respect to $\varphi$ if the following diagram commutes
\[
\begin{CD}
\check*_{i\in \I} (\A_l^{(i)} \check * \A_r^{(i)}) @>\iota>> \A @>\varphi>> \bC \\
@V*_{i\in \I} (\pi_i * \pi_i)VV @. @|\\
\check *_{i\in\I} (B(\H_i) \check * B(\H_i)) @>*_{i\in I} (\lambda_{i} * \rho_{i})>> B(\H) @>\langle \cdot \xi,\xi\rangle >> \bC
\end{CD}
\]
where $\iota$ is the unique $*$-homomorphism extending the identity on each $\A_\chi^{(i)}$, for all $\chi\in \{l,r\}$ and $i\in\I$.
\end{definition}

From this we can now define the main objects of this paper. 

\begin{definition}
Let $(\A_l^{(i)}, \A_r^{(i)})_{i\in \I}$ be a family of pairs of faces in the non-commutative C$^*$-probability space $(\A, \varphi)$. As before, denote $\varphi_i$ to be the restriction of $\varphi$ to $C^*(\A_l^{(i)}, \A_r^{(i)})$ and let $(\pi_i, \H_i, \xi_i)$ be the GNS construction of $(C^*(\A_l^{(i)}, \A_r^{(i)}), \varphi_i)$.

The {\em reduced bi-free product} of $(\A_l^{(i)}, \A_r^{(i)})_{i\in\I}$ with respect to the states $\varphi_i$ is 
\[
({**}_{i\in\I} (\A_l^{(i)}, \A_r^{(i)}), {**}_{i\in\I} \varphi_i) = {**}_{i\in\I} ((\A_l^{(i)}, \A_r^{(i)}), \varphi_i)
\]
which is made up of the unital C$^*$-subalgebra of $B(\H)$, called the {\em reduced bi-free product of C$^*$-algebras},
\[
{**}_{i\in\I} (\A_l^{(i)}, \A_r^{(i)}) := C^*((\lambda_i\circ\pi_i(\A_l^{(i)}), \rho_i\circ\pi_i(\A_r^{(i)}))_{i\in\I}) \subset B(\H)
\]
and the {\em bi-free product state}
\[
{**}_{i\in\I} \varphi_i(\cdot) := \langle \cdot \xi,\xi\rangle.
\]
\end{definition}

It is an immediate fact that the family of pairs of faces $(\lambda_i\circ\pi_i(\A_l^{(i)}), \rho_i\circ\pi_i(\A_r^{(i)}))_{i\in\I}$ is bi-freely independent with respect to the bi-free product state.

It should be noted that we are working within the framework of the original non-commutative C$^*$-probability space $(\A,\varphi)$. This means that the reduced bi-free product is taking into account the behaviour of $\varphi$ not just on the left and right faces but on the C$^*$-algebra they generate, $C^*(\A_l^{(i)}, \A_r^{(i)})$. Since bi-free independence is a statement about the behaviour in the original C$^*$-probability space this definition makes sense.

That being said, one can create the reduced bi-free product as an external product.
Start with pairs of faces in different C$^*$-probability spaces and simply create a new C$^*$-probability space by taking the full free product of the C$^*$-algebras and their associated states and then proceed with the above reduced bi-free product construction.

\section{Faithfulness of bi-free product states}

We first establish what happens when the bi-free product state is faithful.

\begin{theorem}\label{T:mainresult}
Let $(\A_l^{(i)}, \A_r^{(i)})_{i\in \I}$ be a non-trivial family of pairs of faces in the non-commutative $C^*$-probability space $(\A, \varphi)$ such that $\varphi_i = \varphi|_{C^*(\A_l^{(i)}, \A_r^{(i)})}$ is faithful on $C^*(\A_l^{(i)}, \A_r^{(i)})$ for each $i\in \I$. 
If ${**}_{i\in \I} \varphi_i$ is faithful on the reduced bi-free product ${**}_{i\in \I} (\A_l^{(i)}, \A_r^{(i)})$ then
\[
C^*(\A_l^{(i)}, \A_r^{(i)}) \simeq \A_l^{(i)} \otimes_{min} \A_r^{(i)}.
\]
\end{theorem}

\begin{proof}
First we will establish that $\A_l^{(i)}$ and $\A_r^{(i)}$ commute in $\A$, then we will show that they induce a C$^*$-norm on the algebraic tensor product $\A_l^{(i)} \odot \A_r^{(i)}$ and finally that this is in fact the minimal tensor norm.

We will be using the notation from Section 2. To simplify things a little bit, because the $\varphi_i$ are assumed to be faithful, consider $C^*(\A_l^{(i)}, \A_r^{(i)})$ as already a subalgebra of $B(\H_i)$ and so $\varphi_i(\cdot) = \langle \cdot \xi_i,\xi_i\rangle$. That is, we are suppressing the $\pi_i$ notation from the GNS construction. Moreover, we will be using the convention that $\lambda_i(x), \rho_i(x), \lambda_i*\rho_i(x)$ all are living in $B(\H)$.

Suppose $a_\chi \in \A_\chi^{(i)}$ such that $\varphi_i(a_\chi) = 0$, $\chi \in \{l,r\}$ and $0\neq b\in \A_l^{(j)} \cup \A_r^{(j)}$ for $j\neq i$ such that $\varphi_j(b) = 0$. Such a $b$ exists by the non-triviality of the family of pairs of faces.
This gives that $\langle b^*\xi_j,\xi_j\rangle =\varphi_j(b)=0$ and so $b^*\xi_j\in \mathring \H_j$ while $\langle b(b^*\xi_j), \xi_j\rangle = \varphi_j(bb^*) \neq 0$ by the faithfulness of $\varphi_j$.


Now, \cite[Section 1.5]{Voic14} establishes that $[\lambda_i(\A_l^{(i)}), \rho_i(\A_r^{(i)})](\H \ominus \H_i) = 0$
which gives that
\[
(\lambda_i(a_l)\rho_i(a_r)\lambda_j * \rho_j(b) - \rho_i(a_r)\lambda_i(a_l)\lambda_j * \rho_j(b))\xi = 0
\]
 since $b\xi \in \mathring \H_j \subset \H$.
The faithfulness of ${**}_{i\in \I} \varphi_i$ implies that $\xi$ is a separating vector for the reduced bi-free product and thus
\[
\lambda_i(a_l)\rho_i(a_r)\lambda_j * \rho_j(b) - \rho_i(a_r)\lambda_i(a_l)\lambda_j * \rho_j(b) = 0
\] which gives that
\begin{align*}
0 &= P_{\H_i}(\lambda_i(a_l)\rho_i(a_r)\lambda_j * \rho_j(b) - \rho_i(a_r)\lambda_i(a_l)\lambda_j * \rho_j(b))b^*\xi_j
\\& = (\lambda_i(a_l)\rho_i(a_r) - \rho_i(a_r)\lambda_i(a_l))\langle bb^*\xi_j, \xi_j\rangle \xi
\\ & = \langle bb^*\xi_j, \xi_j\rangle(a_la_r - a_ra_l)\xi_i. 
\end{align*}
Since $\xi_i$ is separating for $C^*(\A_l^{(i)}, \A_r^{(i)})$ this implies that $a_l$ and $a_r$ commute. Thus, $\A_l^{(i)}$ and $\A_r^{(i)}$ commute in $\A$ for every $i\in \I$.

\vskip 6 pt
\noindent \textit{Claim:} The canonical map from $\A_l^{(i)} \odot \A_r^{(i)}$ to $C^*(\A_l^{(i)}, \A_r^{(i)})$ is injective. 
\vskip 6 pt
Since $\A_l^{(i)}$ and $\A_r^{(i)}$ commute, the universal property of $\A_l^{(i)} \odot \A_r^{(i)}$ gives that there exists a $*$-homomorphism
\[
\sum_{k=1}^m a_{k,l}\odot a_{k,r} \mapsto \sum_{k=1}^m a_{k,l}a_{k,r}.
\]
We need to establish its injectivity.
 To this end, consider $h\in \mathring \H_j$, $\|h\|=1$ where $j\neq i$ and the isometric map 
\[
V_h : \H_i \otimes \H_i \rightarrow \H_i \otimes h \otimes \H_i
\]
defined by $V_h(h_l \otimes h_r) = h_l \otimes h \otimes h_r$ for $h_l, h_r\in \H_i$. This map is inspired by Dykema's proof of the faithfulness of the free product state \cite[Theorem 1.1]{Dykema}.
Note that in $\H$ we really have that
\[
 \H_i \otimes h \otimes \H_i = \bC h \oplus (\mathring \H_i \otimes h) \oplus (h \otimes \mathring \H_i) \oplus (\mathring \H_i \otimes h \otimes \mathring \H_i)
\]
but hopefully the reader will pardon the simplified notation.

Now $\H_i \otimes h \otimes \H_i$ is a reducing subspace of $C^*(\lambda_i(\A_l^{(i)}), \rho_i(\A_r^{(i)}))$ since for all $a\in \A_l^{(i)}, b\in \A_r^{(i)}$ and $\eta_1, \eta_2 \in \H_i$ we have that
\begin{align*}
V_h^*\lambda_i(a)\rho_i(b)V_h(\eta_1\otimes \eta_2) & = V_h^*\lambda_i(a)\rho_i(b)(\eta_1\otimes h\otimes \eta_2)
\\ & = a\eta_1 \otimes b\eta_2.
\end{align*}
Thus, compressing to $\H_i \otimes h\otimes \H_i$ gives
\[
V_h^*C^*(\lambda_i(\A_l^{(i)}), \rho_i(\A_r^{(i)}))V_h = \A_l^{(i)} \otimes_{min}\A_r^{(i)}.
\]
So, if $\sum_{k=1}^m a_{k,l}\odot a_{k,r} \neq 0 \in \A_l^{(i)} \odot \A_r^{(i)}$ then $\sum_{k=1}^m a_{k,l}\otimes a_{k,r} \neq 0 \in \A_l^{(i)} \otimes_{min}\A_r^{(i)}$ which implies that 
\begin{align*}
0 & \neq \sum_{k=1}^m a_{k,l}\otimes a_{k,r}(\xi_i \otimes \xi_i)
\\ & = V_h^*\sum_{k=1}^m \lambda_i(a_{k,l})\rho_i(a_{k,r})V_h(\xi_i \otimes \xi_i)
\\ & = \sum_{k=1}^m \lambda_i(a_{k,l})\rho_i(a_{k,r})h
\end{align*}
since the state $\langle \cdot \xi_i\otimes \xi_i, \xi_i\otimes \xi_i\rangle$ is faithful on the min tensor product.
But then $\sum_{k=1}^m \lambda_i(a_{k,l})\rho_i(a_{k,r}) \neq 0 \in C^*(\lambda_i(\A_l^{(i)}), \rho_i(\A_r^{(i)}))$ which by the faithfulness of ${**}_{i\in \I} \varphi_i$ gives that $\sum_{k=1}^m \lambda_i(a_{k,l})\rho_i(a_{k,r})\xi \neq 0$. Finally,
\begin{align*}
0 & \neq \left\langle \sum_{k=1}^m \lambda_i(a_{k,l})\rho_i(a_{k,r})\xi,\sum_{k=1}^m \lambda_i(a_{k,l})\rho_i(a_{k,r})\xi\right\rangle
\\ & = \left\langle \left(\sum_{k=1}^m \lambda_i(a_{k,l})\rho_i(a_{k,r})\right)^*\left(\sum_{k=1}^m \lambda_i(a_{k,l})\rho_i(a_{k,r})\right)\xi, \xi\right\rangle
\\ & = \varphi_i\left(\left(\sum_{k=1}^m a_{k,l}a_{k,r}\right)^*\left(\sum_{k=1}^m a_{k,l}a_{k,r}\right)\right)
\end{align*}
which gives by the faithfulness of $\varphi_i$ that $\sum_{k=1}^m a_{k,l}a_{k,r} \neq 0$. Therefore, the claim is verified.

\vskip 6 pt
Now, this implies that $C^*(\A_l^{(i)}, \A_r^{(i)}) \simeq \A_l^{(i)} \otimes_\alpha \A_r^{(i)}$ where $\|\cdot \|_\alpha$ is a C$^*$-norm on $\A_l^{(i)} \odot \A_r^{(i)}$. So by Takesaki's Theorem \cite{Takesaki} we have that there exists a surjective $*$-homomorphism 
\[
q: C^*(\A_l^{(i)}, \A_r^{(i)}) \rightarrow \A_l^{(i)} \otimes_{min} \A_r^{(i)}.
\]
To finish the proof all we need to do is show that $q$ is injective. 

To this end, let $a\in C^*(\A_l^{(i)}, \A_r^{(i)})$ such that $q(a) = 0$. Again as in the first part of this proof, find $0\neq b \in \A_l^{(j)} \cup \A_r^{(j)}$ for $j\neq i$ such that $\varphi_j(b)=0$ and $h\in \mathring \H_j$ such that $\langle bh, \xi_j\rangle \neq 0$. Additionally, assume that $\|b\xi_j\| = 1$.

In the second part of this proof we saw that compressing to $\H_i \otimes b\xi_j \otimes \H_i$ is tantamount to this quotient homomorphism $q$. Namely, suppose 
\[
\iota_i : \A_l^{(i)} \check * \A_r^{(i)} \rightarrow C^*(\A_l^{(i)}, \A_r^{(i)})\  \ (\subseteq B(\H_i) \ \textrm{by assumption})
\] 
is the unique $*$-homomorphism extending the identity in each component.
There then exists $\tilde a \in \A_l^{(i)} \check * \A_r^{(i)}$ such that $\iota_i(\tilde a) = a$. An important fact to record is that, by uniqueness, 
\[
\lambda_i * \rho_i(\cdot)|_{\H_i} = \iota_i(\cdot),
\]
remembering that we have that $\lambda_i * \rho_i(\cdot) \in B(\H)$.
Thus,
\[
V_{b\xi_j}^* \lambda_i * \rho_i(\tilde a) V_{b\xi_j} = q(a) = 0,
\]
which implies, by the fact that $V_{b\xi_j}(\H_i \otimes \H_i)$ is reducing for $\lambda_i * \rho_i(\A_l^{(i)} \check * \A_r^{(i)})$, that 
\begin{align*}
0 &= \lambda_i * \rho_i(\tilde a) V_{b\xi_j}(\xi_i \otimes \xi_i) 
\\ & = \lambda_i * \rho_i(\tilde a)(b\xi_j)
\\ & = \lambda_i * \rho_i(\tilde a)\lambda_j * \rho_j(b)\xi.
\end{align*}
By the faithfulness of the bi-free product state $\lambda_i * \rho_i(\tilde a)\lambda_j * \rho_j(b) = 0$ and so
\begin{align*}
0 & = P_{\H_i}\lambda_i * \rho_i(\tilde a)\lambda_j * \rho_j(b)h
\\ & = \lambda_i * \rho_i(\tilde a)\langle bh, \xi_j\rangle \xi
\\ & = \langle bh,\xi_j\rangle \iota_i(\tilde a)\xi
\\ & = \langle bh, \xi_j\rangle a\xi_i.
\end{align*}
Hence, by the faithfulness of $\varphi_i$ we have that $a=0$. Therefore, for all $i\in \I$, $C^*(\A_l^{(i)}, \A_r^{(i)}) \simeq \A_l^{(i)} \otimes_{min} \A_r^{(i)}$.
\end{proof}

We turn now to a partial converse of the previous theorem. This is probably known among the experts in bi-free probability but we could not find a published proof. The following proof may be a tad clunky but we find it the clearest from a non-expert perspective.

\begin{theorem}\label{T:bifreefaithful}
Let $(\A_l^{(i)}, \A_r^{(i)})_{i\in \I}$ be a family of pairs of faces in the non-commutative C$^*$-probability space $(\A, \varphi)$. If $C^*(\A_l^{(i)}, \A_r^{(i)}) \simeq \A_l^{(i)} \otimes_{min} \A_r^{(i)}$ and $\varphi_i = \varphi_i|_{\A_l^{(i)}} \otimes \varphi_i|_{\A_r^{(i)}}$ is a faithful product state on $C^*(\A_l^{(i)}, \A_r^{(i)})$, for all $i\in \I,$ then ${**}_{i\in\I} \varphi_i$ is faithful on the reduced bi-free product and
\[
{**}_{i\in \I} (\A_l^{(i)}, \A_r^{(i)})_{i\in \I} \ \simeq \ *_{i\in \I} (\A_l^{(i)}, \varphi)  \otimes_{min} *_{i\in \I} (\A_r^{(i)}, \varphi).
\]
\end{theorem}
\begin{proof}
As before, we will be using the notation of Section 2.

For each $i\in\I$, since $C^*(\A_l^{(i)}, \A_r^{(i)}) \simeq \A_l^{(i)} \otimes_{min} \A_r^{(i)}$ and $\varphi_i$ is a product state we can a priori choose $\H_i = \H_{i,l} \otimes \H_{i,r}$, unit vectors $\xi_{i,l}\in \H_{i,l}, \xi_{i,r} \in \H_{i,r}$ such that $\xi_i = \xi_{i,l}\otimes \xi_{i,r}$ and $*$-homomorphisms $\pi_{i,\chi} : \A_\chi^{(i)} \rightarrow B(\H_{i,\chi})$ such that $\pi_i = \pi_{i,l}\otimes \pi_{i,r}$. This will give for $a_\chi\in \A_\chi^{(i)}$, $\chi\in\{l,r\}$, that
\begin{align*}
\varphi_i(a_la_r) & = \langle \pi_i(a_la_r)\xi_i,\xi_i\rangle 
\\ &= \langle \pi_{i,l}(a_l)\xi_{i,l},\xi_{i,l}\rangle \langle \pi_{i,r}(a_r)\xi_{i,r},\xi_{i,r}\rangle. 
\end{align*}
Along with the free product Hilbert space 
\[
(\H, \xi) = *_{i\in \I} (\H_i, \xi_{i})
\]
we need to also define, for $\chi \in \{l,r\}$, the free product Hilbert spaces
\[
(\H_\chi, \xi_\chi) = *_{i\in \I} (\H_{i,\chi}, \xi_{i,\chi}).
\]
Since there are multiple free product Hilbert spaces we will use subscripts to denote the different left and right representations, namely, 
\[
\lambda_{\H_i} : B(\H_i) \rightarrow B(\H) \quad \textrm{and} \quad  \lambda_{\H_{i,l}}: B(\H_{i,l}) \rightarrow B(\H_l)
\]
for the left representations and
\[
\rho_{\H_i}: B(\H_i) \rightarrow B(\H) \quad \textrm{and} \quad  \rho_{\H_{i,r}} : B(\H_{i,r}) \rightarrow B(\H_r)
\]
for the right representations.

Dykema's original result \cite{Dykema} proves that $\langle \cdot \xi_\chi, \xi_\chi\rangle$ is faithful on $*_{i\in \I} (\A_\chi^{(i)}, \varphi)$ for $\chi\in \{l,r\}$ and it is a folklore result that the minimal tensor product of faithful states is faithful. Thus, $\langle \cdot\: \xi_l \otimes \xi_r, \xi_l \otimes \xi_r \rangle$ is faithful on $*_{i\in \I} (\A_l^{(i)}, \varphi) \ \otimes_{min} \; *_{i\in \I} (\A_r^{(i)}, \varphi)$.

Fix $k\geq1$ and $j_1,\cdots, j_k \in \I$ such that $j_i \neq j_{i+1}, 1\leq i\leq k-1$. Now fix a unit vector 
\begin{align*}
h &= (\xi_{j_1,l} \otimes h_{j_1,r}) \otimes h_{j_2} \otimes \cdots \otimes h_{j_{k-1}} \otimes (h_{j_k, l} \otimes \xi_{j_k,r})
\\ & \in (\xi_{j_1, l} \otimes \mathring \H_{j_1, r}) \otimes \mathring \H_{j_2} \otimes \cdots \otimes \mathring \H_{j_{k-1}} \otimes (\mathring \H_{j_k, l} \otimes \xi_{j_k, r}).
\end{align*}
If $k=1$ the only possible $h$ is $\xi = \xi_{j_1} = \xi_{j_1,l} \otimes \xi_{j_1,r}$.
Call the collection of such $h$, as $k$ and the indices vary, $\S \subset \H$. 

As will be shown below, this set of unit vectors $\S$ plays an important role in decomposing simple tensors in $\H$, in particular for every simple tensor $\eta \in \H$ that is also a simple tensor in each component there exists a unique $h\in \S$ such that $\eta \in \H_l \otimes h\otimes \H_r$. By abuse of tensor notation this is not very hard to see in one's mind but the reality of proving this carefully needs plenty of indices.

To this end, for $m\geq 1$ suppose $s_1,\dots, s_m\in \I$ such that $s_t \neq s_{t+1}$ for $1\leq t\leq m-1$, and $\eta_{t, l} \in \H_{s_t,l}$, $\eta_{t,r}\in \H_{s_t, r}$ such that $\eta_{t,l}\otimes \eta_{t,r} \in \mathring \H_{s_t}$ for $1\leq t\leq m$. This last condition implies that $\eta_{t,\chi}=\|\eta_{t,\chi}\|\xi_{t,\chi}$ cannot hold for both $\chi=l$ and $\chi=r$.
In summary,
\[
\eta := (\eta_{1,l}\otimes \eta_{1,r})\otimes \cdots \otimes (\eta_{m,l}\otimes \eta_{m,r}) \in \mathring \H_{s_1} \otimes \cdots \otimes \mathring \H_{s_m}.
\]
Note that the conditions imposed on the $\eta_{t,\chi}$ in the above paragraph imply that the form of $\eta$ above is as reduced as it can be.

As mentioned above, it will be established that there exists $h\in \S$ such that 
\[
\eta  \in \H_l \otimes h \otimes \H_r.
\]
To prove the required decomposition, 
let 
\[
v = \max\{0\leq t\leq m \ : \ \eta_{j,r} = \|\eta_{j,r}\|\xi_{s_j,r}, \ 1\leq j\leq t\}
\]
and 
\[
w = \min\{1\leq t\leq m+1 \ : \ \eta_{j,l} = \|\eta_{j,l}\|\xi_{s_j,l}, \ t\leq j\leq m\}.
\]
This gives that $v$ is the number of terms in a row from the left with trivial right tensor components and $m+1 - w$ is the number of terms in a row from the right with trivial left tensor components.

By the fact that $\eta_{t,l}\otimes \eta_{t,r} \in \mathring \H_{s_t}$, that is $\eta_{t,\chi}=\|\eta_{t,\chi}\|\xi_{t,\chi}$ cannot hold for both $\chi=l$ and $\chi=r$, we have that $v<w$.
If $v=m$ then $w=m+1$ and $\eta\in \H_l$, and if $w=1$ then $v=0$ and $\eta\in \H_r$.
Otherwise, when $0\leq v\leq m-1$ and $2\leq w\leq m$, define
\begin{align*}
\eta_l &= (\eta_{1,l}\otimes \|\eta_{1,r}\|\xi_{s_1,r}) \otimes \cdots \otimes (\eta_{v,l} \otimes \|\eta_{v,r}\|\xi_{s_v,r}) \otimes (\eta_{v+1,l} \otimes \xi_{s_{v+1},r}),
\\ \eta_\S &= (\xi_{v+1,l} \otimes \eta_{v+1,r}) \otimes \cdots \otimes (\eta_{w-1,l} \otimes \xi_{w-1,r}),
\\ \eta_r &= (\xi_{w-1,l} \otimes \eta_{w-1,r}) \otimes (\|\eta_{w,l}\|\xi_{w,l} \otimes \eta_{w,r}) \otimes \cdots \otimes (\|\eta_{m,l}\|\xi_{m,l} \otimes \eta_{m,r})
\end{align*}
with $\eta_\S = \xi$ if $v+1=w$. Hence, by the usual slight abuse of the tensor notation, $\eta = \eta_l \otimes \eta_\S \otimes \eta_r \in \H_l \otimes \eta_\S \otimes \H_r$ with $\frac{1}{\|\eta_\S\|}\eta_\S \in \S$.
Therefore,
\[
\overline\spn\{ \H_l \otimes h \otimes \H_r : h\in \S\} = \H.
\]

For any $h\in\S$, which is a unit vector, there is a natural isometric map $S_h : \H_l \otimes \H_r \rightarrow \H$ given by the concatenation $\H_l \otimes \H_r \mapsto \H_l \otimes h \otimes \H_r$ with the appropriate simplification of tensors when needed. In particular, there exist $k\geq1$ and $j_1,\cdots, j_k \in \I$ such that $j_i \neq j_{i+1}, 1\leq i\leq k-1$ and then
\begin{align*}
h &= (\xi_{j_1,l} \otimes h_{j_1,r}) \otimes h_{j_2} \otimes \cdots \otimes h_{j_{k-1}} \otimes (h_{j_k, l} \otimes \xi_{j_k,r})
\\ & \in (\xi_{j_1, l} \otimes \mathring \H_{j_1, r}) \otimes \mathring \H_{j_2} \otimes \cdots \otimes \mathring \H_{j_{k-1}} \otimes (\mathring \H_{j_k, l} \otimes \xi_{j_k, r}).
\end{align*}
We can now carefully specify that the isometric map is given by
\[
\xi_l \otimes \xi_r  \mapsto h,
\]
\begin{gather*}
\xi_l \otimes (\mathring\H_{i_1,r}\otimes \cdots \otimes \mathring\H_{i_m, r}) \to 
\\   \left\{\begin{array}{ll} h \otimes (\xi_{i_1,l} \otimes \mathring\H_{i_1,r})\otimes \cdots \otimes (\xi_{i_m,l} \otimes \mathring\H_{i_m, r}),  & i_1\neq j_k
\\ (\xi_{j_1,l} \otimes h_{j_1,r}) \otimes h_{j_2} \otimes \cdots \otimes h_{j_{k-1}} \otimes (h_{j_k, l} \otimes \mathring\H_{i_1,r}) \otimes
\\ \hspace{190 pt} \cdots \otimes (\xi_{i_m,l} \otimes \mathring\H_{i_m, r}), & i_1 = j_k
\end{array}\right. 
\end{gather*}
\begin{gather*} 
(\mathring\H_{i_1,l}\otimes \cdots \otimes \mathring\H_{i_m, l}) \otimes \xi_r  \to 
\\ \left\{\begin{array}{ll} (\mathring\H_{i_1,l}\otimes \xi_{i_1, r})\otimes \cdots \otimes (\mathring\H_{i_m, l} \otimes \xi_{i_m,r}) \otimes h, & i_m\neq j_1
\\ (\mathring\H_{i_1,l}\otimes \xi_{i_1, r})\otimes \cdots \otimes (\mathring\H_{i_m, l} \otimes h_{j_1,r}) \otimes h_{j_2} \otimes 
\\ \hspace{160.5 pt} \cdots \otimes h_{j_{k-1}} \otimes (h_{j_k, l} \otimes \xi_{j_k,r}) , & i_m= j_1 \end{array}\right.
\end{gather*}
and 
\begin{gather*}
(\mathring\H_{i_1,l}\otimes \cdots \otimes \mathring\H_{i_m, l}) \otimes (\mathring\H_{t_1,r}\otimes \cdots \otimes \mathring\H_{t_s, r}) \to
\\ (\mathring\H_{i_1,l}\otimes \xi_{i_1,r}) \otimes \cdots \otimes (\mathring\H_{i_m, l} \otimes \xi_{i_m,r})\otimes h \otimes (\xi_{t_1,l}\otimes \mathring\H_{t_1,r})\otimes \cdots \otimes (\xi_{t_s, l} \otimes \mathring\H_{t_s, r})
\end{gather*}
if $i_m \neq j_1$ and $j_k \neq t_1$
with similar statements as the cases above when $i_m=j_1$ or $j_k = t_1$ or both happen.
Perhaps the most natural case of $S_h$ is when $h=\xi$. It certainly minimizes, but doesn't remove, the need for all of the cases above.

A careful examination of the $S_h$ isometric map implies that for $a\in \A_l^{(i_1)}$, $b\in \A_r^{(i_2)}$ and $\eta_\chi \in \H_\chi$ for $\chi\in\{l,r\}$ we have that, by abuse of the tensor notation,
\begin{align*}
\lambda_{\H_{i_1}}(\pi_{i_1}(a))&\rho_{\H_{i_2}}(\pi_{i_2}(b))S_h(\eta_l \otimes \eta_r) 
\\ & = \lambda_{\H_{i_1}}(\pi_{i_1}(a))\rho_{\H_{i_2}}(\pi_{i_2}(b))(\eta_l \otimes h\otimes \eta_r)
\\ & = \lambda_{\H_{i_1,l}}(\pi_{i_1,l}(a))\eta_l \otimes h \otimes \rho_{\H_{i_2,r}}(\pi_{i_2,r}(b))\eta_r
\\ & = S_h(\lambda_{\H_{i_1,l}}(\pi_{i_1,l}(a))\eta_l \otimes \rho_{\H_{i_2,r}}(\pi_{i_2,r}(b))\eta_r)
\end{align*}
Hence, $S_h(\H_l \otimes \H_r)$ is a reducing subspace of the reduced bi-free product. Moreover,
\[
S_h^*\lambda_{\H_i}\circ\pi_i(\cdot) S_h = (\lambda_{\H_{i,l}}\circ\pi_{i,l}(\cdot)) \otimes I_{\H_r} \quad \quad \textrm{on}\ \A_l^{(i)}
\]
and 
\[
S_h^*\rho_{\H_i}\circ\pi_i(\cdot) S_h = I_{\H_l}\otimes(\rho_{\H_{i,r}}\circ\pi_{i,r}(\cdot)) \quad \quad \textrm{on}\ \A_r^{(i)}
\]
Therefore, for any $h\in \S$,
\[
S_h^*({**}_{i\in \I} (\A_l^{(i)}, \A_r^{(i)}))S_h = *_{i\in \I} (\A_l^{(i)}, \varphi) \ \otimes_{min} \; *_{i\in \I} (\A_r^{(i)}, \varphi)
\]
and furthermore, by the identities involving $S_h,\lambda$ and $\rho$, $S_h^*aS_h = S_\xi^*(a)S_\xi$ for all $h\in \S$ and $a\in {**}_{i\in \I} (\A_l^{(i)}, \A_r^{(i)})$.

Finally, we want to show that compression to $S_\xi(\H_l \otimes \H_r)$ is a $*$-isomorphism. Note that this is the same as compression to $S_h(\H_l\otimes \H_r)$ being injective for any $h\in \S$. This gives us a way forward.
Suppose that $a\in {**}_{i\in \I} (\A_l^{(i)}, \A_r^{(i)})$ such that ${**}_{i\in\I} \varphi_i(a^*a) = 0$. This implies that 
\begin{align*}
0 &= a\xi
\\ & = S_\xi^*aS_\xi (\xi_l \otimes \xi_r).
\end{align*}
By the faithfulness of $\langle \cdot\: \xi_l \otimes \xi_r, \xi_l \otimes \xi_r\rangle$ this gives that $S_\xi^* a S_\xi = 0$ or rather $a$ is 0 on the reducing subspace $S_\xi(\H_l \otimes \H_r)$. But then for all $h\in \S$ we have that 
\[
S_h^*aS_h = S_\xi^*a S_\xi = 0
\]
and $a$ is 0 on the reducing subspace $S_h(\H_l \otimes \H_r)$. By what we proved about the set $\S$, we have that $a$ is 0 on 
\[
\overline\spn\{ S_h(\H_l \otimes \H_r) : h\in \S\} = \overline\spn\{ \H_l \otimes h \otimes \H_r : h\in \S\} = \H.
\]
Therefore, $a=0$ and thus ${**}_{i\in\I} \varphi_i$ is faithful.
\end{proof}

There may exist a full converse to Theorem \ref{T:mainresult} but the previous proof highly depends on the state $\varphi_i$ arising as a tensor product of states. In general, $\varphi_i$ need not be of this form. We should note here that if $\varphi_i|_{\A_l^{(i)}}$ or $\varphi_i|_{\A_l^{(i)}}$ is a pure state then $\varphi_i$ will be a tensor product of states.

To end this paper, we summarize with the following corollary.

\begin{corollary}
Let $(\A_l^{(i)}, \A_r^{(i)})_{i\in \I}$ be a non-trivial family of pairs of faces in the non-commutative C$^*$-probability space $(\A, \varphi)$. If $\varphi$ is faithful on $C^*((\A_l^{(i)}, \A_r^{(i)})_{i\in\I})$, $C^*(\A_l^{(i)}, \A_r^{(i)}) \simeq \A_l^{(i)} \otimes_{min} \A_r^{(i)}$, $\varphi_i = \varphi_i|_{\A_l^{(i)}} \otimes \varphi_i|_{\A_r^{(i)}}$ and $(\A_l^{(i)}, \A_r^{(i)})_{i\in \I}$ is bi-freely independent with respect to $\varphi$, then
\begin{align*}
C^*((\A_l^{(i)}, \A_r^{(i)})_{i\in \I}) & \simeq {**}_{i\in \I} (\A_l^{(i)}, \A_r^{(i)})_{i\in \I}
\\ & \simeq *_{i\in \I} (\A_l^{(i)}, \varphi)  \otimes_{min} *_{i\in \I} (\A_r^{(i)}, \varphi).
\end{align*}
\end{corollary}
\begin{proof}
Recall, that by bi-free independence we know that the following diagram commutes
\[
\begin{CD}
\check*_{i\in \I} (\A_l^{(i)} \check * \A_r^{(i)}) @>\iota>> C^*((\A_l^{(i)}, \A_r^{(i)})_{i\in \I}) @>\varphi>> \bC \\
@V*_{i\in \I} (\pi_i * \pi_i)VV @. @|\\
\check *_{i\in\I} (B(\H_i) \check * B(\H_i)) @>*_{i\in I} (\lambda_{i} * \rho_{i})>> B(\H) @>\langle \cdot \xi,\xi\rangle >> \bC
\end{CD}
\]
Because both of the states are faithful on their algebras then for any $a^*a \in \check*_{i\in \I} (\A_l^{(i)} \check * \A_r^{(i)})$, $a^*a$ is in the kernel of $\iota$ if and only if $a^*a$ is in the kernel of $*_{i\in I} (\lambda_{i} * \rho_{i})\circ *_{i\in \I} (\pi_i * \pi_i)$. Therefore, both quotients are $*$-isomorphic and Theorem \ref{T:bifreefaithful} gives the final $*$-isomorphism.
\end{proof}


\end{document}